\documentclass[11pt]{amsart}

\usepackage[margin=1.08in]{geometry}
\usepackage{amsmath,amssymb,amsthm,mathtools}
\usepackage[T1]{fontenc}
\usepackage[utf8]{inputenc}
\usepackage{lmodern}
\usepackage{microtype}
\usepackage[dvipsnames]{xcolor}
\usepackage{hyperref}
\usepackage{enumitem}

\hypersetup{
  colorlinks=true,
  linkcolor=blue!45!black,
  citecolor=blue!45!black,
  urlcolor=blue!45!black,
  pdftitle={On the sequence gcd(a to the n minus 1, b to the n minus 1)},
  pdfauthor={Khai-Hoan Nguyen-Dang}
}

\newtheorem{theorem}{Theorem}[section]
\newtheorem{proposition}[theorem]{Proposition}
\newtheorem{corollary}[theorem]{Corollary}
\newtheorem{lemma}[theorem]{Lemma}
\theoremstyle{remark}

\numberwithin{equation}{section}
\allowdisplaybreaks

\newcommand{\Z}{\mathbb Z}
\newcommand{\Q}{\mathbb Q}
\newcommand{\ord}{\operatorname{ord}}

\title{On the sequence $\gcd(a^n-1,b^n-1)$}
\author[K.-H. Nguyen-Dang]{Khai-Hoan Nguyen-Dang}
\address{Morningside Center of Mathematics, Chinese Academy of Sciences, No.~55, Zhongguancun East Road, Beijing 100190, China}
\email{khaihoann@gmail.com}
\date{}

\begin{document}

\begin{abstract}
For integers $a,b\ge 2$, put
\[
  g_n:=\gcd(a^n-1,b^n-1)\qquad(n\ge 1).
\]
We prove that $(g_n)$ satisfies a constant-coefficient linear recurrence if and only if
$a$ and $b$ are multiplicatively dependent; the same conclusion holds if only a tail
of $(g_n)$ is assumed to satisfy such a recurrence.  More generally, if $a$ and $b$
are multiplicatively independent, then for every fixed shift $N\ge 0$, every
integer sequence $(W_n)_{n\ge 1}$ that satisfies a constant-coefficient linear
recurrence and the pointwise divisibilities
\[
  W_n\mid a^{n+N}-1
  \qquad\text{and}\qquad
  W_n\mid b^{n+N}-1
  \qquad(n\ge 1)
\]
is eventually periodic.  If $(W_n)$ satisfies a recurrence with nonzero trailing
coefficient, then it is periodic from the first term.  In particular, every common
integer linear divisibility-sequence factor of $a^n-1$ and $b^n-1$, in the convention
adopted here, is periodic.  The proof combines the Bugeaud--Corvaja--Zannier bound
with the structure of subexponential linear recurrences and a local valuation
argument.
\end{abstract}

\subjclass[2020]{Primary 11B37, 11A05; Secondary 11B39, 11D61}
\keywords{gcd-sequence, linear recurrence sequence, linear divisibility sequence,
Ailon--Rudnick conjecture}
\maketitle

\section{Introduction}\label{sec:introduction}

For integers $a,b\ge 2$, define
\[
  g_n:=\gcd(a^n-1,b^n-1)\qquad(n\ge 1).
\]
The sequence $(g_n)$ is automatically a divisibility sequence: if $m\mid n$, then
\[
  a^m-1\mid a^n-1
  \qquad\text{and}\qquad
  b^m-1\mid b^n-1,
\]
hence $g_m\mid g_n$.  The sequence $(g_n)$ therefore lies at the intersection
of two classical themes: greatest common divisors of exponential sequences and
the structure theory of divisibility sequences.

The modern study of small values of $g_n$ begins with Ailon and Rudnick
\cite{AR04}.  They proved the following strong polynomial analogue over
$\mathbf C[T]$: if $f,g\in\mathbf C[T]$ are multiplicatively independent, then
there exists a nonzero polynomial $h\in\mathbf C[T]$ such that
\[
  \gcd(f^n-1,g^n-1)\mid h
  \qquad(n\ge 1).
\]
In the integer setting they conjectured that, if $a$ and $b$ are
multiplicatively independent and
\[
  \gcd(a-1,b-1)=1,
\]
then
\[
  \gcd(a^n-1,b^n-1)=1
\]
for infinitely many $n$.  The fundamental unconditional quantitative result in
this direction is the theorem of Bugeaud, Corvaja, and Zannier \cite{BCZ03}:
for multiplicatively independent integers $a,b\ge 2$ and every
$\varepsilon>0$,
\[
  \gcd(a^n-1,b^n-1)<\exp(\varepsilon n)
\]
for all sufficiently large $n$; equivalently,
\[
  \log \gcd(a^n-1,b^n-1)=o(n).
\]
Levin \cite{Lev19} established higher-dimensional $S$-unit and algebraic-torus
gcd inequalities, together with applications to simple linear recurrence
sequences.  Grieve and Wang \cite{GW20} developed moving-target versions and
further applications to algebraic linear recurrence sequences.  See also the
survey of Tron \cite{Tron20}.

The problem also sits naturally inside Silverman's generalized-gcd program
\cite{Sil05}, which interprets gcd questions through local heights and
intersections on blowups and places divisibility sequences in the setting of
algebraic groups.  Several related function-field and geometric extensions have
subsequently been developed.  Silverman studied common divisors of elliptic
divisibility sequences over function fields \cite{Sil04}; Ostafe proved
univariate and multivariate extensions of the polynomial Ailon--Rudnick theorem
\cite{Ost16}; Ghioca, Hsia, and Tucker obtained an elliptic-surface variant,
including a special case of a conjecture of Silverman \cite{GHT18}; and
Barroero, Capuano, and Turchet proved function-field results for abelian and
split semiabelian varieties \cite{BCT24}.  Silverman has also formulated
explicit rational and number-field versions of the original question
\cite{Sil20}; we discuss what carries over, and the additional obstacles, in
Subsection~\ref{subsec:number-field}.

The purpose of the present paper is to determine exactly what constant-coefficient
linear recurrence methods and linear divisibility-sequence methods can and cannot
explain about the sequence $(g_n)$ itself.  More precisely, a sequence
$(u_n)_{n\ge 1}\subset\Z$ is a \emph{constant-coefficient linear recurrence
sequence} if there exist an integer $d\ge 1$ and integers $c_1,\dots,c_d$ such
that
\begin{equation}\label{eq:recurrence-definition}
  u_{n+d}=c_1u_{n+d-1}+\cdots+c_du_n
  \qquad(n\ge 1).
\end{equation}
Such sequences are also called \emph{$C$-finite}.  Here the trailing coefficient
$c_d$ is allowed to vanish.  We say that the sequence admits a recurrence with
\emph{nonzero trailing coefficient} if a relation of the form
\eqref{eq:recurrence-definition} can be chosen with $c_d\ne 0$.  This condition
excludes a zero characteristic root and permits the recurrence to be propagated
backwards; without it, a finite initial transient may occur.  If $d$ is minimal
among relations of the form \eqref{eq:recurrence-definition} valid from the first
index, then $d$ is called the \emph{order} of the recurrence.  Throughout the
paper, the unqualified term \emph{linear recurrence} means a
constant-coefficient linear recurrence over $\Z$.

If
\[
  U(x):=\sum_{n\ge 1}u_nx^n,
\]
then
\[
  \bigl(1-c_1x-\cdots-c_dx^d\bigr)U(x)\in\Z[x].
\]
Hence
\[
  U(x)=\frac{A(x)}{B(x)}
\]
for some $A(x)\in\Z[x]$ and
\[
  B(x)=1-c_1x-\cdots-c_dx^d\in\Z[x],
\]
so in particular $B(0)=1$.

A sequence $(u_n)_{n\ge 1}$ of nonzero integers is a \emph{divisibility
sequence} if $u_m\mid u_n$ whenever $m\mid n$.  Following the recurrence
convention used by Granville \cite{Gra22}, we call it a \emph{linear divisibility
sequence} (LDS) if it also admits a constant-coefficient recurrence with nonzero
trailing coefficient.  Granville uses the synonymous expression \emph{linear
division sequence}.

Classical work of Hall and Ward initiated the systematic study of linear
divisibility sequences \cite{Hall36,Ward37}.  B\'ezivin, Peth\H{o}, and van der
Poorten obtained a far-reaching structural description of divisibility sequences
among linear recurrences \cite{BPV90}, and Granville has now classified integer
linear division sequences in full generality \cite{Gra22}.  Our contribution lies
exactly at the boundary of this theory: $(g_n)$ is always a divisibility sequence,
but it is a constant-coefficient linear recurrence if and only if $a$ and $b$ are
multiplicatively dependent.  More generally, when $a$ and $b$ are multiplicatively
independent and $N\ge 0$ is fixed, every integer linear recurrence that divides
both $a^{n+N}-1$ and $b^{n+N}-1$ termwise is eventually periodic; it is periodic
from the first term whenever it admits a recurrence with nonzero trailing
coefficient.  In particular, every common integer LDS factor of $a^n-1$ and
$b^n-1$ is periodic.

Our first result gives the exact classification of the gcd-sequence itself.

\begin{theorem}\label{thm:classification}
Let $a,b\ge 2$ be integers, and let
\[
  g_n=\gcd(a^n-1,b^n-1)\qquad(n\ge 1).
\]
The following are equivalent:
\begin{enumerate}[label=\textup{(\arabic*)},leftmargin=2.5em]
  \item $a$ and $b$ are multiplicatively dependent;
  \item $(g_n)_{n\ge 1}$ satisfies a constant-coefficient linear recurrence;
  \item for some $N\ge 0$, the tail $(g_{n+N})_{n\ge 1}$ satisfies a
  constant-coefficient linear recurrence.
\end{enumerate}
If these conditions hold, then there are integers $c\ge 2$ and $r,s\ge 1$ such
that $a=c^r$ and $b=c^s$.  Writing $d=\gcd(r,s)$, one has
\[
  g_n=c^{dn}-1\qquad(n\ge 1).
\]
\end{theorem}

The second result is stronger than the LDS statement suggested by the divisibility
of $(g_n)$: every integer linear recurrence that is a pointwise common factor is
eventually periodic, even after a fixed shift in the exponent.  A nonzero trailing
coefficient removes the possible finite transient.

\begin{theorem}\label{thm:common-factor}
Assume that $a$ and $b$ are multiplicatively independent, and fix $N\ge 0$.
Let $(W_n)_{n\ge 1}$ be an integer sequence satisfying a constant-coefficient
linear recurrence such that
\begin{equation}\label{eq:common-factor-hypothesis}
  W_n\mid a^{n+N}-1
  \qquad\text{and}\qquad
  W_n\mid b^{n+N}-1
  \qquad(n\ge 1).
\end{equation}
Then $(W_n)_{n\ge 1}$ is eventually periodic.  If $(W_n)$ satisfies a recurrence
with nonzero trailing coefficient, then it is periodic from the first term.
\end{theorem}

Taking $N=0$ shows, in particular, that every common integer LDS factor of
$a^n-1$ and $b^n-1$ is periodic.  Notice that Theorem~\ref{thm:common-factor}
does not assume the divisibility-sequence property for $(W_n)$.

The distinction between eventual and pure periodicity is unavoidable under the broad
$C$-finite convention.  For example, with $a=3$ and $b=5$, the sequence
$W_1=1$ and $W_n=2$ for $n\ge 2$ divides both $a^n-1$ and $b^n-1$ at every
index; it is eventually periodic and $C$-finite, but it is not periodic from the first
term.  It does not admit a recurrence with nonzero trailing coefficient.

\subsection{Rational and number-field variants}\label{subsec:number-field}

Silverman explicitly formulated rational and number-field versions of the
Ailon--Rudnick question \cite{Sil20}.  For $x,y\in\Q^\times$, let
$\gcd_{\Q}(x,y)$ denote the gcd of the numerators of $x$ and $y$ in lowest terms.
Silverman's rational version asks whether, for multiplicatively independent
$a,b\in\Q^\times$, one has
\[
  \gcd_{\Q}(a^n-1,b^n-1)=\gcd_{\Q}(a-1,b-1)
\]
for infinitely many $n$.  If $a=A/C$ in lowest terms, then the numerator of
$a^n-1$ is $A^n-C^n$.  Consequently, much of the present argument survives over
$\Q$: the relevant numerator sequences are still constant-coefficient linear
recurrences, the local lifting-the-exponent estimates have the same form away from
the finitely many primes dividing the numerators and denominators, and the global
subexponential input is supplied by the corresponding $S$-unit gcd estimates
\cite{Lev19,GW20}.

For a number field $K$, Silverman proposes the scalar quantity
\[
  \gcd_K(\alpha,\beta)
  :=\gcd_{\Q}\bigl(N_{K/\Q}(\alpha),N_{K/\Q}(\beta)\bigr),
  \qquad \alpha,\beta\in K^\times,
\]
and asks for the analogous equality between the gcd at exponent $n$ and the gcd at
exponent $1$ for infinitely many $n$.  A more local alternative is the gcd ideal
\[
  \mathfrak g_n=(\alpha^n-1,\beta^n-1)\mathcal O_{K,S}
\]
for a finite set $S$ containing the primes needed to make $\alpha$ and $\beta$
$S$-units.  The global small-gcd input has substantial number-field and
algebraic-torus extensions \cite{Sil05,Lev19,GW20}, but the recurrence argument does
not transfer verbatim.  There is no canonical gcd element when $\mathfrak g_n$ is
nonprincipal, and an ideal-valued sequence has no immediate analogue of a scalar
constant-coefficient recurrence.  Passing to $N_{K/\Q}(\mathfrak g_n)$ merges the
contributions of distinct prime ideals, and the resulting integer sequence is not
automatically a linear recurrence even though the ambient exponential sequences
are.  One must also distinguish multiplicative dependence in $K^\times$ from
versions taken modulo torsion, and keep track of roots of unity and $S$-units.
Finally, the elementary use of Fermat's theorem in the proof below must be replaced
by residue-field order conditions; obtaining suitable prime ideals uniformly may
require Chebotarev-type distribution input.  Thus the global Diophantine estimate
has an analogue, while the principal additional obstacles are the passage from ideal
divisibility to a scalar recurrence and the control of local factors uniformly across
prime ideals.

\subsection*{Organization}
Section~\ref{sec:preliminaries} records the elementary common-base identity, the
Bugeaud--Corvaja--Zannier theorem, and the local valuation estimate used later.
Section~\ref{sec:subexponential} proves that a subexponential integer linear
recurrence is eventually polynomial on residue classes and establishes the local
polynomial lemma needed to rule out nonconstant pieces.  In
Section~\ref{sec:main-proofs} we prove Theorems~\ref{thm:common-factor}
and~\ref{thm:classification}.

\subsection*{Acknowledgements}
The author thanks Morningside Center of Mathematics, Chinese Academy of Sciences,
for its support and a stimulating research environment.  The author also thanks
Professor Andrew Granville for comments on an initial draft and the referee for a
careful report and helpful suggestions that substantially improved the exposition.

\section{Preliminaries}\label{sec:preliminaries}

We say that positive integers $a$ and $b$ are \emph{multiplicatively dependent} if
$a^u=b^v$ for some positive integers $u,v$; otherwise they are
\emph{multiplicatively independent}.

\begin{lemma}[Common base and gcd identity]\label{lem:common-base}
Let $a,b\ge 2$ be multiplicatively dependent.  Then there are integers
$c\ge 2$ and $r,s\ge 1$ such that
\[
  a=c^r,\qquad b=c^s.
\]
Moreover, for every $x\ge 2$ and all positive integers $m,n$,
\[
  \gcd(x^m-1,x^n-1)=x^{\gcd(m,n)}-1.
\]
\end{lemma}

\begin{proof}
Choose positive $u,v$ with $a^u=b^v$, and divide both exponents by
$\gcd(u,v)$, so that $\gcd(u,v)=1$.  For every prime $p$,
\[
  u\,v_p(a)=v\,v_p(b).
\]
Hence $v\mid v_p(a)$ and $u\mid v_p(b)$.  Writing
$v_p(a)=v e_p$ and $v_p(b)=u e_p$ and putting $c=\prod_p p^{e_p}$ gives
$a=c^v$ and $b=c^u$.

For the gcd identity, put $d=\gcd(m,n)$.  The number $x^d-1$ divides both
$x^m-1$ and $x^n-1$.  Conversely, let $D>1$ be a common divisor.  Then
$\gcd(D,x)=1$, and the multiplicative order of $x$ modulo $D$ divides both
$m$ and $n$.  It therefore divides $d$, so $D\mid x^d-1$.  (The case $D=1$
is immediate.)
\end{proof}

The global input is the following theorem of Bugeaud, Corvaja, and Zannier
\cite[Theorem~1]{BCZ03}.

\begin{theorem}[Bugeaud--Corvaja--Zannier]\label{thm:BCZ}
Let $a,b\ge 2$ be multiplicatively independent integers.  Then, for every
$\varepsilon>0$, there is $n_0=n_0(a,b,\varepsilon)$ such that
\[
  \gcd(a^n-1,b^n-1)<\exp(\varepsilon n)
\]
for every $n\ge n_0$.
\end{theorem}

We also need an elementary local estimate.  Here $v_p$ denotes the usual
$p$-adic valuation, applied to the absolute value of a nonzero integer.

\begin{lemma}[Local valuation bound]\label{lem:valuation-bound}
Let $c>1$ be an integer and let $p$ be a prime.  There is a constant
$C_{c,p}$ such that
\[
  v_p(c^m-1)\le C_{c,p}+v_p(m)
  \qquad(m\ge 1).
\]
\end{lemma}

\begin{proof}
If $p\mid c$, then $v_p(c^m-1)=0$ for all $m$.  Assume $p\nmid c$.
If $p$ is odd, let $t=\ord_p(c)$.  The valuation is zero unless $t\mid m$.
If $m=tk$, then $p\nmid t$, and the lifting-the-exponent formula gives
\[
  v_p(c^m-1)=v_p(c^t-1)+v_p(k)
            =v_p(c^t-1)+v_p(m).
\]
If $p=2$, then $c$ is odd.  The same formula in its $2$-adic form says
\[
  v_2(c^m-1)=v_2(c-1)\quad\text{if $m$ is odd},
\]
and
\[
  v_2(c^m-1)=v_2(c-1)+v_2(c+1)+v_2(m)-1
  \quad\text{if $m$ is even}.
\]
These formulas give the asserted bound.
\end{proof}

\section{Subexponential recurrences and local polynomial structure}
\label{sec:subexponential}

We first recall the standard consequence of Kronecker's theorem that turns a
subexponential integer recurrence into an eventual quasipolynomial.

\begin{lemma}[Kronecker]\label{lem:kronecker}
Let $\alpha\ne 0$ be an algebraic integer.  If every Galois conjugate of
$\alpha$ has absolute value at most $1$, then $\alpha$ is a root of unity.
\end{lemma}

\begin{proof}
Let $d=[\Q(\alpha):\Q]$.  For every $m\ge 1$, the algebraic integer
$\alpha^m$ has degree at most $d$, and all its conjugates have absolute value at
most $1$.  The coefficients of its minimal polynomial are therefore integers
bounded solely in terms of $d$.  Only finitely many such polynomials occur, and each has finitely many roots;
therefore only finitely many values $\alpha^m$ occur.  Hence
$\alpha^m=\alpha^n$ for some $m>n$.  Since $\alpha\ne 0$, it follows that
$\alpha^{m-n}=1$.
\end{proof}

\begin{proposition}[Subexponential recurrences]\label{prop:quasipolynomial}
Let $(u_n)_{n\ge 1}\subset\Z$ be a sequence satisfying a constant-coefficient
linear recurrence, and assume that
\[
  \log^+|u_n|=o(n).
\]
Then there are an integer $M\ge 1$ and polynomials
$Q_0,\dots,Q_{M-1}\in\Q[X]$ such that, for every residue class
$r\pmod M$,
\[
  u_n=Q_r(n)
\]
for all sufficiently large $n\equiv r\pmod M$.
\end{proposition}

\begin{proof}
The generating series $U(x)=\sum_{n\ge 1}u_n x^n$ is rational.  Starting from
the recurrence denominator, whose constant term is $1$, and applying Gauss's lemma
after cancellation, we may write
\[
  U(x)=\frac{A(x)}{B(x)},
  \qquad A,B\in\Z[x],\quad \gcd(A,B)=1,\quad B(0)=1.
\]
If $B$ is constant, then $B=1$ and $U$ is a polynomial, so $u_n=0$ for all
sufficiently large $n$; the conclusion follows with $M=1$ and $Q_0=0$.  We may
therefore assume that $B$ is nonconstant.  The growth hypothesis implies that the
radius of convergence is at least $1$; hence every zero $\rho$ of $B$ satisfies
$|\rho|\ge 1$.  The reciprocal $\alpha=\rho^{-1}$ is an algebraic integer,
because the reciprocal polynomial of $B$ is monic.  Every conjugate of $\alpha$
is the reciprocal of a zero of $B$, so Lemma~\ref{lem:kronecker} shows that
$\alpha$ is a root of unity.

Partial fractions therefore give, for all sufficiently large $n$,
\[
  u_n=\sum_{j=1}^s\sum_{k=1}^{e_j}
       c_{j,k}\binom{n+k-1}{k-1}\zeta_j^n,
\]
where the $\zeta_j$ are roots of unity.  Let $M$ be a common multiple of their
orders.  On each residue class $r\pmod M$, the factors $\zeta_j^n$ are
constant, so $u_n=Q_r(n)$ eventually for some $Q_r\in\mathbf C[X]$.

It remains to see that $Q_r$ has rational coefficients.  This is immediate if
$Q_r=0$.  Otherwise, choose $N_r$ so that
$Q_r(r+M(t+N_r))\in\Z$ for all $t\ge 0$, and put
\[
  P_r(t):=Q_r(r+M(t+N_r)).
\]
Then $P_r(t)\in\Z$ for every $t\ge 0$.  If $d=\deg P_r$, Newton interpolation
gives
\[
  P_r(t)=\sum_{j=0}^d \Delta^j P_r(0)\binom{t}{j},
\]
where every $\Delta^j P_r(0)$ is an integer.  Thus $P_r\in\Q[t]$, and hence
$Q_r\in\Q[X]$.
\end{proof}

The next lemma is the local step that replaces the need for a general
classification theorem for linear divisibility sequences.  Its hypothesis says that
fixed-prime valuations can grow only through the factor $n+N$.

\begin{lemma}[Valuation-constrained quasipolynomials]\label{lem:local-polynomial}
Fix $N\ge 0$.  Let $(u_n)_{n\ge 1}$ be a sequence of nonzero integers.  Suppose
that there are $M\ge 1$ and $Q_0,\dots,Q_{M-1}\in\Q[X]$ such that
\[
  u_n=Q_r(n)
\]
for all sufficiently large $n\equiv r\pmod M$.  Suppose also that for every
prime $p$ there is a constant $C_p$ satisfying
\begin{equation}\label{eq:valuation-hypothesis}
  v_p(u_n)\le C_p+v_p(n+N)
  \qquad(n\ge 1).
\end{equation}
Then, for each $r\pmod M$, there are $\lambda_r\in\Q^\times$ and
$e_r\in\{0,1\}$ such that
\[
  Q_r(X)=\lambda_r(X+N)^{e_r}.
\]
\end{lemma}

\begin{proof}
Fix $r$ and set $Q=Q_r$.  The polynomial $Q$ is nonzero, because $u_n$ is
nonzero on infinitely many integers in the residue class.  Choose $D\ge 1$ such
that
\[
  F(Y):=D\,Q(Y-N)\in\Z[Y].
\]
Write
\[
  F(Y)=Y^e H(Y),
  \qquad H(0)\ne 0.
\]
We first show that $H$ is constant.

Suppose otherwise, and choose a primitive nonconstant irreducible factor
$h\in\Z[Y]$ of $H$.  By Schur's theorem, infinitely many primes divide a
nonzero value of $h$.  We may therefore choose such a prime $p$ outside the
finite set of primes dividing
\[
  D M h(0)\operatorname{disc}(h)
\]
and the leading coefficient of $h$.  Thus $h$ has a nonzero simple root modulo
$p$.  By the simple-root form of Hensel's lemma, for every $k\ge 1$ there is a
nonzero residue class $y_k\pmod{p^k}$ such that
\[
  h(y_k)\equiv 0\pmod{p^k}.
\]
Since $p\nmid M$, the Chinese remainder theorem gives arbitrarily large integers
$y$ satisfying
\[
  y\equiv r+N\pmod M,
  \qquad
  y\equiv y_k\pmod{p^k}.
\]
Put $n=y-N$.  Taking $y$ large ensures $n\ge 1$ and that the eventual polynomial
identity applies.  Moreover $p\nmid y=n+N$, whereas
$p^k\mid F(y)=D u_n$.  Since $p\nmid D$, this yields
$v_p(u_n)\ge k$, contradicting \eqref{eq:valuation-hypothesis} for large $k$.
Therefore $H$ is constant, and
\[
  Q(X)=\lambda(X+N)^e
\]
for some $\lambda\in\Q^\times$ and $e\ge 0$.

It remains to prove $e\le 1$.  Choose a prime $p\nmid DM$ and outside the
numerator and denominator of $\lambda$.  For every $k\ge 1$, the Chinese
remainder theorem gives arbitrarily large $y$ such that
\[
  y\equiv r+N\pmod M,
  \qquad
  y\equiv p^k\pmod{p^{k+1}}.
\]
For $n=y-N$ in the eventual range, one has $v_p(n+N)=k$ and
$v_p(u_n)=ek$.  Hence \eqref{eq:valuation-hypothesis} gives
$ek\le C_p+k$ for every $k$, so $e\le 1$.
\end{proof}

For completeness, we record the elementary passage from eventual periodicity to
periodicity for recurrences with nonzero trailing coefficient.

\begin{lemma}\label{lem:eventual-periodic}
Let $(u_n)_{n\ge 1}$ satisfy a constant-coefficient recurrence with nonzero trailing
coefficient.  If $(u_n)$ is eventually periodic, then it is periodic.
\end{lemma}

\begin{proof}
Let $M$ be an eventual period and set $v_n=u_{n+M}-u_n$.  Choose a recurrence
\eqref{eq:recurrence-definition} for $(u_n)$ with $c_d\ne 0$.  The sequence $(v_n)$
satisfies the same recurrence and vanishes for all sufficiently large $n$.  Moreover,
\[
  c_d v_n=v_{n+d}-c_1v_{n+d-1}-\dots-c_{d-1}v_{n+1}.
\]
Since $c_d\ne 0$, consecutive zeros propagate backwards.  Thus $v_n=0$ for every
$n\ge 1$, and $M$ is a period of the whole sequence.
\end{proof}

\section{Proofs of the main theorems}\label{sec:main-proofs}

\begin{proof}[Proof of Theorem~\ref{thm:common-factor}]
Because $W_n$ divides the two nonzero integers in
\eqref{eq:common-factor-hypothesis}, one has $W_n\ne 0$ and
\[
  |W_n|\le \gcd(a^{n+N}-1,b^{n+N}-1).
\]
Applying Theorem~\ref{thm:BCZ} at the exponent $n+N$ gives
\[
  \log^+|W_n|=o(n+N)=o(n).
\]
By Proposition~\ref{prop:quasipolynomial}, there are $M\ge 1$ and polynomials
$Q_0,\dots,Q_{M-1}\in\Q[X]$ such that $W_n=Q_r(n)$ for all sufficiently large
$n\equiv r\pmod M$.

For every prime $p$, Lemma~\ref{lem:valuation-bound} gives
\[
  v_p(W_n)\le v_p(a^{n+N}-1)
  \le C_{a,p}+v_p(n+N).
\]
Lemma~\ref{lem:local-polynomial} now shows that
\begin{equation}\label{eq:linear-or-constant}
  Q_r(X)=\lambda_r(X+N)^{e_r},
  \qquad e_r\in\{0,1\},
\end{equation}
for every $r\pmod M$.

We claim that in fact $e_r=0$ for every $r$.  Suppose that $e_r=1$ for some
$r\in\{0,\dots,M-1\}$.  Put
\[
  d:=\gcd(r+N,M),\qquad r+N=d r_0,\qquad M=dM_0.
\]
Then $\gcd(r_0,M_0)=1$, except in the harmless case $r+N=0$, when $M_0=1$.
By Dirichlet's theorem, there are infinitely many primes
\[
  \ell\equiv r_0\pmod{M_0}.
\]
Write $\lambda_r=A_r/B_r$ in lowest terms, with
$A_r\in\Z\setminus\{0\}$ and $B_r\in\Z_{>0}$.  Choose such primes outside the
finite set dividing $a d A_rB_r$, and large enough that
\[
  n_\ell:=d\ell-N\ge 1
\]
lies in the eventual range.  Since $n_\ell\equiv r\pmod M$,
\eqref{eq:linear-or-constant} gives
\[
  W_{n_\ell}=\frac{A_r d\ell}{B_r}\in\Z.
\]
Because $\gcd(A_r,B_r)=1$ and $\ell\nmid B_r$, integrality forces
$B_r\mid d$.  Consequently
\[
  W_{n_\ell}=\left(\frac{A_r d}{B_r}\right)\ell,
\]
so $\ell\mid W_{n_\ell}$.  The divisibility hypothesis implies
\[
  \ell\mid a^{d\ell}-1.
\]
As $\ell\nmid a$, Fermat's theorem yields
$a^{d\ell}\equiv a^d\pmod\ell$, and hence
\[
  \ell\mid a^d-1.
\]
This is impossible for infinitely many distinct primes $\ell$.  Thus every
$e_r=0$.

Consequently, $(W_n)$ is eventually constant on each residue class modulo $M$,
so it is eventually periodic.  If $(W_n)$ satisfies a recurrence with nonzero
trailing coefficient, Lemma~\ref{lem:eventual-periodic} shows that it is periodic
from the beginning.
\end{proof}

\begin{proof}[Proof of Theorem~\ref{thm:classification}]
Assume first that $a$ and $b$ are multiplicatively dependent.  By
Lemma~\ref{lem:common-base}, there are $c\ge 2$ and $r,s\ge 1$ such that
$a=c^r$ and $b=c^s$.  If $d=\gcd(r,s)$, then
\[
  g_n=\gcd(c^{rn}-1,c^{sn}-1)=c^{dn}-1.
\]
In particular,
\[
  g_{n+2}=(c^d+1)g_{n+1}-c^d g_n,
\]
so $(g_n)$ satisfies a constant-coefficient linear recurrence.  This proves
\textup{(1)$\Rightarrow$(2)}, and \textup{(2)$\Rightarrow$(3)} is immediate.

It remains to prove \textup{(3)$\Rightarrow$(1)}.  Suppose, to the contrary, that
$a$ and $b$ are multiplicatively independent and that
\[
  W_n:=g_{n+N}
\]
satisfies a constant-coefficient linear recurrence for some $N\ge 0$.  Since
\[
  W_n\mid a^{n+N}-1
  \qquad\text{and}\qquad
  W_n\mid b^{n+N}-1,
\]
Theorem~\ref{thm:common-factor} shows that $(W_n)$ is eventually periodic, and
hence bounded.
On the other hand, for every sufficiently large prime $\ell\nmid ab$, the index
$n=\ell-1-N$ is positive and Fermat's theorem gives
\[
  \ell\mid a^{\ell-1}-1,
  \qquad
  \ell\mid b^{\ell-1}-1.
\]
Thus
\[
  \ell\mid W_{\ell-1-N}=g_{\ell-1},
\]
and hence $|W_{\ell-1-N}|\ge \ell$.  Therefore $(W_n)$ is unbounded.  This
contradiction proves that $a$ and $b$ are
multiplicatively dependent.
\end{proof}

\begin{corollary}\label{cor:lds-classification}
The sequence $(g_n)_{n\ge 1}$ is a linear divisibility sequence if and only if
$a$ and $b$ are multiplicatively dependent.
\end{corollary}

\begin{proof}
The sequence $(g_n)$ is always a divisibility sequence, so the assertion follows
immediately from Theorem~\ref{thm:classification}.
\end{proof}


\begin{thebibliography}{99}

\bibitem{AR04}
N.~Ailon and Z.~Rudnick,
\emph{Torsion points on curves and common divisors of $a^k-1$ and $b^k-1$},
Acta Arith. \textbf{113} (2004), no.~1, 31--38.
\href{https://doi.org/10.4064/aa113-1-3}{doi:10.4064/aa113-1-3}.

\bibitem{BCT24}
F.~Barroero, L.~Capuano, and A.~Turchet,
\emph{Greatest common divisor results on semiabelian varieties and a conjecture of Silverman},
Res. Number Theory \textbf{10} (2024), Paper No.~17.
\href{https://doi.org/10.1007/s40993-023-00494-2}{doi:10.1007/s40993-023-00494-2}.

\bibitem{BCZ03}
Y.~Bugeaud, P.~Corvaja, and U.~Zannier,
\emph{An upper bound for the G.C.D. of $a^n-1$ and $b^n-1$},
Math. Z. \textbf{243} (2003), no.~1, 79--84.
\href{https://doi.org/10.1007/s00209-002-0449-z}{doi:10.1007/s00209-002-0449-z}.

\bibitem{BPV90}
J.-P.~B\'ezivin, A.~Peth\H{o}, and A.~J.~van der Poorten,
\emph{A full characterisation of divisibility sequences},
Amer. J. Math. \textbf{112} (1990), no.~6, 985--1001.
\href{https://doi.org/10.2307/2374733}{doi:10.2307/2374733}.

\bibitem{GHT18}
D.~Ghioca, L.-C.~Hsia, and T.~J.~Tucker,
\emph{A variant of a theorem by Ailon--Rudnick for elliptic curves},
Pacific J. Math. \textbf{295} (2018), no.~1, 1--15.
\href{https://doi.org/10.2140/pjm.2018.295.1}{doi:10.2140/pjm.2018.295.1}.

\bibitem{Gra22}
A.~Granville,
\emph{Classifying linear division sequences},
Amer. J. Math., to appear;
\href{https://arxiv.org/abs/2206.11823}{arXiv:2206.11823}.

\bibitem{GW20}
N.~Grieve and J.~Tzu-Yueh Wang,
\emph{Greatest common divisors with moving targets and consequences for linear recurrence sequences},
Trans. Amer. Math. Soc. \textbf{373} (2020), no.~11, 8095--8126.
\href{https://doi.org/10.1090/tran/8220}{doi:10.1090/tran/8220}.

\bibitem{Hall36}
M.~Hall,
\emph{Divisibility sequences of third order},
Amer. J. Math. \textbf{58} (1936), no.~3, 577--584.
\href{https://doi.org/10.2307/2370976}{doi:10.2307/2370976}.

\bibitem{Lev19}
A.~Levin,
\emph{Greatest common divisors and Vojta's conjecture for blowups of algebraic tori},
Invent. Math. \textbf{215} (2019), no.~2, 493--533.
\href{https://doi.org/10.1007/s00222-018-0831-z}{doi:10.1007/s00222-018-0831-z}.

\bibitem{Ost16}
A.~Ostafe,
\emph{On some extensions of the Ailon--Rudnick theorem},
Monatsh. Math. \textbf{181} (2016), no.~2, 451--471.
\href{https://doi.org/10.1007/s00605-016-0911-3}{doi:10.1007/s00605-016-0911-3}.

\bibitem{Sil04}
J.~H.~Silverman,
\emph{Common divisors of elliptic divisibility sequences over function fields},
Manuscripta Math. \textbf{114} (2004), no.~4, 431--446.
\href{https://doi.org/10.1007/s00229-004-0468-7}{doi:10.1007/s00229-004-0468-7}.

\bibitem{Sil05}
J.~H.~Silverman,
\emph{Generalized greatest common divisors, divisibility sequences, and Vojta's conjecture for blowups},
Monatsh. Math. \textbf{145} (2005), no.~4, 333--350.
\href{https://doi.org/10.1007/s00605-005-0299-y}{doi:10.1007/s00605-005-0299-y}.

\bibitem{Sil20}
J.~H.~Silverman,
\emph{A problem in (ostensibly elementary) number theory},
Problem~225 in M.~T.~Rassias, \emph{Solved and unsolved problems},
EMS Newsl., no.~115 (2020), 45--52, at 47--48.
\href{https://doi.org/10.4171/NEWS/115/11}{doi:10.4171/NEWS/115/11}.

\bibitem{Tron20}
E.~Tron,
\emph{The greatest common divisor of linear recurrences},
Rend. Semin. Mat. Univ. Politec. Torino \textbf{78} (2020), no.~1, 103--124.
\href{https://seminariomatematico.polito.it/rendiconti/78-1.html}{journal page}.

\bibitem{Ward37}
M.~Ward,
\emph{Linear divisibility sequences},
Trans. Amer. Math. Soc. \textbf{41} (1937), no.~2, 276--286.
\href{https://doi.org/10.1090/S0002-9947-1937-1501902-1}{doi:10.1090/S0002-9947-1937-1501902-1}.

\end{thebibliography}
\end{document}